\documentclass[]{interact}

\usepackage{epstopdf}
\usepackage[caption=false]{subfig}

\usepackage[numbers,sort&compress]{natbib}
\bibpunct[, ]{[}{]}{,}{n}{,}{,}
\renewcommand\bibfont{\fontsize{10}{12}\selectfont}
\makeatletter
\def\NAT@def@citea{\def\@citea{\NAT@separator}}
\makeatother

\theoremstyle{plain}
\newtheorem{theorem}{Theorem}[section]
\newtheorem{lemma}[theorem]{Lemma}

\newtheorem{proposition}[theorem]{Proposition}

\theoremstyle{definition}
\newtheorem{definition}[theorem]{Definition}
\newtheorem{example}[theorem]{Example}

\theoremstyle{remark}
\newtheorem{remark}{Remark}

\allowdisplaybreaks

\newtheorem*{assumption*}{\assumptionnumber}
\providecommand{\assumptionnumber}{}
\makeatletter
\newenvironment{assumption}[2]
{%
	\renewcommand{\assumptionnumber}{Assumption #1#2}%
	\begin{assumption*}%
		\protected@edef\@currentlabel{#1#2}%
	}
	{%
	\end{assumption*}
}
\usepackage{xcolor}

\usepackage{csquotes}
\MakeOuterQuote{"}

\begin{document}

\title{Solution Properties of Convex Sweeping Processes with Velocity Constraints}
\author{
	\name{N.~N. Thieu\textsuperscript{$\dagger$,$\ddagger$}\thanks{CONTACT N.~N. Thieu. Email: nang-thieu.nguyen@unilim.fr; nguyennangthieu@gmail.com}}
	\affil{\textsuperscript{$\dagger$}Laboratoire XLIM, Universit\'e de Limoges, 87060 Limoges, France; \textsuperscript{$\ddagger$}Institute of Mathematics, Vietnam Academy of Science and Technology, Hanoi, Vietnam	}
}

\maketitle
	
\begin{abstract}
	Some properties of solutions of convex sweeping processes with velocity constraints are studied in this paper. Namely, the solution sensitivity with respect to the initial value, the boundedness, the closedness, and the convexity of the solution set are discussed in detail. Our investigations complement the preceding ones on the solution existence and the solution uniqueness of convex sweeping processes with velocity constraints.
\end{abstract}	
	

\begin{amscode}
	49J40  $\cdot$ 47J20 $\cdot$ 47J22  $\cdot$ 58E35  $\cdot$ 34G25
\end{amscode}

\section{Introduction}\label{section_SP_intro}

Sweeping processes with velocity constraints, which are nontrivial generalizations of certain evolution variational inequalities, were studied firstly by Siddiqi and Manchada~\cite{siddiqimanchanda2002}. Since these models have various applications in mechanics, physics, and engineering (see~\cite[p.~8]{aht} and~\cite[Section~6.4]{duvautlions1976}), several forms of such processes have been considered in the literature; see~\cite{aht,ak18,adlyhaddad2020, ATY_2021,bounkhel2007,JouraniVilches2019,vilchesnguyen2020}. 

In this paper, we study the form of sweeping processes with velocity constraints proposed by Adly, Haddad and Thibault~\cite{aht}, which is defined as follows. Suppose that~${\mathcal H}$ is a real Hilbert space, $T$ a positive real number, and $C:[0,T]\rightrightarrows {\mathcal H}$ a set-valued map having nonempty closed convex values. Let $A_0,A_1:{\mathcal H}\to {\mathcal H}$ be positive semidefinite, bounded, symmetric linear operators and $f:[0,T]\to {\mathcal H}$ be a continuous mapping. Consider the following differential inclusion, which is called sweeping process with velocity constraint:
\begin{equation}\label{main}
\left\{
\begin{array}{l}
A_1 \dot{u}(t)+A_0 u(t)-f(t) \in -{\mathcal N}_{C(t)}(\dot{u}(t))\quad \text{a.e.} \; t \in [0,T],\\ 
u(0)=u_0,
\end{array}\right. \tag{\text{P}}
\end{equation}
where ${\mathcal N}_{C(t)}(\dot{u}(t))$ is the normal cone to $C(t)$ at $\dot{u}(t)$ in the sense of convex analysis. An \textit{absolutely continuous} function $u:[0,T]\to{\mathcal H}$ is said to be a \textit{solution} of~\eqref{main} if it satisfies the conditions stated in the formulation of  problem~\eqref{main}. Since the Hilbert space $\mathcal{H}$ has the Radon-Nikod\'ym property, the Fr\'echet derivative $\dot{u}(t)$ of $u$ exists for almost every $t\in [0,T]$  (see Section~\ref{section_SP_Preliminaries} below for some relevant references).

For sweeping processes with velocity in a moving bounded convex set in separable Hilbert spaces, Adly, Haddad and Thibault~\cite{aht} have established sufficient conditions for the solution existence and the solution uniqueness.  Later, Adly and Le~\cite{ak18} have generalized the solution existence result of~\cite{aht} to the case where the moving set can be unbounded and the operator $A_1$ is semicoercive. An application to non-regular electrical circuits was given in~\cite{aht, ak18}. In a subsequent paper, by weakening the continuity condition of the moving constraint set, Vilches and Nguyen~\cite[Section~5]{vilchesnguyen2020} have obtained a refinement of the corresponding result of~\cite{ak18}. For implicit sweeping processes of a general type, Jourani and Vilches~\cite{JouraniVilches2019} have proved the solution existence and uniqueness by using the concept of quasistatic evolution variational inequalities from~\cite{adlyhaddad2018}. 

Relaxing the convexity of the constraint sets, Bounkhel~\cite{bounkhel2007} have obtained the solution existence and the solution uniqueness for~\eqref{main}, provided that ${\mathcal N}_{C(t)}(\dot{u}(t))$ is replaced by the proximal normal cone ${\mathcal N}^P_{C(t)}(\dot{u}(t))$, $A_0\equiv 0$, $A_1$ is the identity operator, and $C(t)$ are uniformly prox-regular for all $t\in [0,T]$. Recently, by using a result of Yen~\cite{Yen_AMO95} on the solution sensitivity of  parametric variational inequalities, Adly, Thieu and Yen~\cite{ATY_2021} have investigated the problem~\eqref{main} in the case where the set-valued mapping $t\mapsto C(t)$, $t\in [0, T]$, is locally Lipschitz-like. The authors have also established several solution existence results for the case where $C(t)$ is a finite union of disjoint convex sets. More comments and remarks on the solution existence and solution uniqueness of~\eqref{main} in both convex and nonconvex cases can be found in~\cite{ATY_2021}. 

Our aim is to study some fundamental properties of the solutions of~\eqref{main}. When the problem has a unique solution, it is of interest to study the continuity of the solution with respect to the initial value $u_0$. We prove that if the sufficient conditions for the solution existence and uniqueness either in~\cite{aht} or in~\cite{ATY_2021} are satisfied, then the solution is Lipschitz continuous on the initial value. Then, we show that the solution set is bounded if some assumptions used in~\cite{aht,ak18,ATY_2021} are fulfilled. The solution set is not always closed in the space of continuous vector-valued functions. However, it is a closed subset in an appropriate space. Two sets of sufficient conditions for the convexity of the solution set are obtained. Interestingly, a sharp outer estimate for the solution set can be established. It is worthy to stress that the just-mentioned properties of the solutions of~\eqref{main} are investigated here for the first time. To the best of our knowledge, analogous results are not available in the literature.

The paper is organized as follows. Some preliminaries, including a lemma on a relation between strong convergence of sequence of functions in $L^1([0,T],{\mathcal H})$ and its pointwise convergence, are presented in Section~\ref{section_SP_Preliminaries}. The solution sensitivity with respect to the initial value is addressed in Section~\ref{section_SP_solutionsensitivity}. Three theorems on the boundedness of the solution set are proved in Section~\ref{section_SP_boundedness}. We establish in Section~\ref{section_SP_closedness} the closedness of the solution set of~\eqref{main} in the Sobolev space $W^{1,1}([0,T],{\mathcal H})$. Section~\ref{section_SP_convexity} is devoted to the convexity of the solution set, an outer estimate for the set, and two interesting open questions. The obtained results are summarized in the last section.

\section{Preliminaries}\label{section_SP_Preliminaries}

Throughout this paper, let $\mathcal{H}$ be a real Hilbert space equipped with the norm $ \Vert \cdot \Vert $ and the scalar product $\langle \cdot,\cdot \rangle$. The open ball (resp., closed ball)  in $ \mathcal{H} $ with center~$x$ and radius $r>0$ is denoted by $ \mathbb{B}(x, r) $ (resp., $\bar{\mathbb{B}}(x, r)$). The distance from $x$ to $\Omega$ is $d(x,\Omega):= \inf\limits_{y\in \Omega}\Vert x-y\Vert.$ The \textit{projection} of a point $x\in \mathcal{H}$ onto $\Omega$ is defined by $\mathbb{P}_\Omega(x) = \big\{y\in \Omega \mid d(x,\Omega)=\Vert x-y \Vert\big\}.$ 
The \textit{Hausdorff distance} between nonempty subsets $\Omega_1$, $\Omega_2$ of $\mathcal{H}$ is defined by the formula
$$d_H(\Omega_1, \Omega_2)=\max\left\{ \sup\limits_{x\in \Omega_1}d(x,\Omega_2) , \;\; \sup\limits_{y\in \Omega_2}d(y,\Omega_1)    \right\}.$$ For a convex set $\Omega\subset \mathcal{H}$, the normal cone to  $\Omega$ at $x\in \mathcal{H}$ in the sense of convex analysis is ${\mathcal N}_{\Omega}(x):=\{x^*\in\mathcal{H}\mid \langle x^*,y-x\rangle\leq 0, \ \forall y\in\Omega\}$ if $x\in\Omega$ and $\emptyset$ if $x\notin\Omega$.

By $\mathbb N$ we denote the set of positive integers. The notation $[a,b]$ (resp., $(a,b)$) stands for a closed interval (resp., an open interval) in the real line $\mathbb R$. The Banach space of continuous functions defined on $[a,b]$ with values in ${\mathcal H}$ is denoted by $\mathcal{C}^0([a,b],{\mathcal H})$. Here, $\Vert x \Vert_{\mathcal{C}^0} = \max\limits_{t\in[a,b]} \Vert x(t)\Vert$. 

\begin{definition}{\rm
		A function $x:[a,b]\to\mathcal{H}$ is said to be \textit{absolutely continuous} on $[a,b]$ if for every $\varepsilon >0$ there is a $\delta >0$ such that $\sum_{k=1}^n \Vert x(b_k)-x(a_k)\Vert <\varepsilon$
		for any finite system of pairwise disjoint subintervals $(a_k,b_k)\subset[a,b]$ of total length $\sum_{k=1}^n(b_k-a_k)$ less than $\delta$.}
\end{definition}

Any absolutely continuous function $u:[0,T]\to\mathcal{H}$ is Fr\'echet differentiable almost everywhere on $[0,T]$ with respect to the Lebesgue measure of the segment; see, for example,~\cite[Corollary~13 of Chapter 3, Theorem~2 on p.~107, and Section 6 of Chapter VII]{diestel1977} or~\cite[Corollary~5.12 and Theorem~5.21]{benyamini1998}. 

\begin{definition}\label{def_Lipschitz-likeness}{\rm (See \cite[Definition~1.40]{Mordukhovich_2006} and \cite[Definition~3.1]{Mordukhovich_2018})
		One says that a set-valued mapping $K:\Lambda\rightrightarrows\mathcal{H}$, where $\Lambda$ is a metric space, is \textit{Lipschitz-like} around a point $(\bar{\lambda},\bar{x})$ in its \textit{graph}, which is the set $\{(\lambda,x)\in\Lambda\times {\mathcal H}\mid x\in K(\lambda)\},$ if there exist a neighborhood $V$ of $\bar\lambda$, a neighborhood $W$ of $\bar{x}$ and a constant $\kappa > 0$ such that
		\begin{equation*}
		K(\lambda)\cap W \subset K(\lambda')+\kappa d(\lambda,\lambda')\bar{\mathbb{B}}(0,1), \quad \forall \lambda,\lambda'\in V.
		\end{equation*}}
\end{definition}

\begin{definition}
	{\rm A linear operator $A:{\mathcal H}\to{\mathcal H}$ is \textit{coercive} if there exists a positive constant $c$ such that \begin{equation}\label{coercivity}\langle Ax,x\rangle \geq c\Vert x\Vert^2\quad \forall x\in {\mathcal H}.
	\end{equation}}
\end{definition}

If there is $c>0$ such that~\eqref{coercivity} holds, then $c\Vert x\Vert^2\leq \langle Ax,x\rangle \leq \Vert A\Vert \Vert x\Vert^2$ for all $x\in {\mathcal H}$. Thus, we must have $c\leq \Vert A\Vert$, provided that ${\mathcal H}\neq\{0\}$. Let $A$ be bounded and coercive. Set \begin{equation}\label{modulus_coercivity}\bar{c}=\sup \left\{c\in\mathbb{R}_+\mid  \langle Ax,x\rangle \geq c\Vert x\Vert^2\ \;\forall  x\in {\mathcal H}\right\}.
\end{equation} By the definition of supremum, there exists a sequence $\{c_k\}\subset\mathbb{R}_+$ satisfying the inequality $\langle Ax,x\rangle \geq c_k\Vert x\Vert^2$ for all $x\in {\mathcal H}$ and $c_k\to\bar{c}$ as $k\to\infty$. Hence, one has $\langle Ax,x\rangle \geq \bar{c}\Vert x\Vert^2$ for all $x\in\mathcal{H}$. So, $\bar{c}\in(0,\Vert A\Vert]$. For a bounded coercive linear operator $A:{\mathcal H}\to{\mathcal H}$, the constant $\bar{c}$ defined by~\eqref{modulus_coercivity} is called the \textit{modulus of coercivity} of $A$.

We now recall the definition of Bochner integral.

\begin{definition}\label{Bochner_integral}	{\rm (See \cite[pp.~44--45]{diestel1977}) Let $(\Omega, \Sigma, \mu)$ be a finite measurable space and $X$ be a Banach space. A $\mu$-measurable function $f:\Omega\to X$ is called \textit{Bochner integrable} if there is a sequence of simple functions $\{f_k\}$ such that
		$\displaystyle\lim\limits_{k\to\infty}\int_\Omega \Vert f_k(\omega)-f(\omega)\Vert_X d\mu =0.$
		In this case, $\displaystyle\int_E f(\omega)d\mu$ is defined for each $E\in\Sigma$ by $\displaystyle\int_E f(\omega)d\mu=\lim\limits_{k\to\infty}\displaystyle\int_E f_k(\omega)d\mu$, where $\displaystyle\int_E f_k(\omega)d\mu$ is defined in an obvious way.}
\end{definition}

As noted in~\cite[p.~45]{diestel1977}, the limit in Definition~\ref{Bochner_integral} exists and is independent of the defining sequence $\{f_k\}$. According to \cite[Theorem~2, p.~45]{diestel1977}, a $\mu$-measurable function $f:\Omega\to X$ is Bochner integrable if and only if $\displaystyle\int_\Omega \Vert f(\omega)\Vert_X d\mu <\infty.$ 

If $u:[0,T]\to\mathcal{H}$ is an absolutely continuous function, then the function $\dot{u}(\cdot)$ is Bochner integrable on $[0,T]$ (see the proof of \cite[Theorem~2, p.~107]{diestel1977} for detailed explanations).

For every $p\in [1,\infty)$, the Bochner space $L^p(\Omega, X)$ consists of all $\mu$-measurable functions $f:\Omega\to X$ satisfying $$\|f\|_p=\left(\int_\Omega \Vert f(\omega)\Vert^p_X d\mu \right)^{1/p}<\infty$$ (see, e.g., \cite[pp.~49--50]{diestel1977}). The space $L^p(\Omega, X)$ for any $1\leq p<\infty$ is a Banach space and the set of simple functions is dense in $L^p(\Omega, X)$ (see, e.g., \cite[p.~97]{diestel1977}). 

The following lemma gives a relation between strong convergence of sequence of functions in $L^1([0,T],{\mathcal H})$ and its pointwise convergence.

\begin{lemma}\label{stronglp_pointwise}
	Let $\{x_n\}$ be a sequence in $L^1([0,T],{\mathcal H})$ and let $x\in L^1([0,T],{\mathcal H})$ be such that $x_n$ converges strongly to $x$ in $L^1([0,T],{\mathcal H})$. Then, there exists a subsequence $\{x_{n_k}\}$ of $\{x_n\}$ such that $x_{n_k}(t)$ converges to $x(t)$ almost everywhere on $[0,T]$.
\end{lemma}
\begin{proof}
	Since $\{x_n\}$ is a strongly convergent sequence, it is a Cauchy sequence. Hence, for every positive integer $k$ we can find a positive integer $n_k$ such that 
	$$\Vert x_{m} -x_{q}\Vert_{L^1} \leq \frac{1}{2^k}\quad (\forall m\geq n_k, \forall q\geq n_k).$$ 
	Without loss of generality we may assume that $n_{k_1}<n_{k_2}$ whenever $k_1<k_2$. Clearly, the above choice of $\{n_k\}$ implies that $\{x_{n_k}\}$ is a subsequence of  $\{x_n\}$ having the property
	\begin{equation}\label{special_subsequence} \Vert x_{n_{k+1}} -x_{n_k}\Vert_{L^1} \leq \frac{1}{2^k} \quad \forall k \geq 1.
	\end{equation}
	Define 
	\begin{equation}\label{y_m} y_m(t)= \sum\limits_{k=1}^m \Vert x_{n_{k+1}}(t)-x_{n_k}(t)\Vert.
	\end{equation} For all $t\in [0,T]$, by~\eqref{y_m} and~\eqref{special_subsequence} we have
	\begin{equation*}
	\vert y_m(t) \vert= \sum\limits_{k=1}^m \Vert x_{n_{k+1}}(t)-x_{n_k}(t)\Vert\leq \sum\limits_{k=1}^m \frac{1}{2^k} \leq 1.
	\end{equation*} 
	Thus, $\vert y_m(t) \vert\leq 1$  for every $t\in [0,T]$. Since $x_n\in L^1([0,T],{\mathcal H})$ is measurable for all $n\in\mathbb{N}$, the function $y_m:[0,T]\to\mathbb R$ is also measurable for all $m\in\mathbb{N}$. As $\{y_m\}$ is an increasing sequence of real-valued functions, by the monotone convergence theorem~\cite[Theorem~4.1]{brezis2011} one can assert that $y_m(t)$ converges to a function $y(t)$ almost everywhere on $[0,T]$. Since $\vert y(t)\vert \leq 1$ for all $t\in[0,T]$, we see that $y\in L^1([0,T],\mathbb{R})$. On the other hand, for $i>j\geq 2$, we have
	\begin{equation}\label{special_estimates} \Vert x_{n_i}(t)-x_{n_j}(t)\Vert \leq \Vert x_{n_i}(t)-x_{n_{i-1}}(t)\Vert+\ldots+\Vert x_{n_{j+1}}(t)-x_{n_j}(t)\Vert \leq y(t)-y_{n_{j-1}}(t).\end{equation}
	It follows that, for almost every $t\in[0,T]$, $\{x_{n_k}(t)\}$ is a Cauchy sequence in ${\mathcal H}$ and it converges to a finite limit, say, $\tilde{x}(t)$. From~\eqref{special_estimates}, letting $i$ tend to infinity, we obtain 
	$$\Vert \tilde{x}(t)-x_{n_j}(t)\Vert \leq y(t)-y_{n_{j-1}}(t)\leq y(t)$$
	for almost every $t\in[0,T]$ and for any $j \geq 2$. 
	Hence, one has $\tilde{x}\in L^1([0,T],{\mathcal H})$. Since $\Vert x_{n_k}(t) -\tilde{x}(t)\Vert^2\to 0$ and $\Vert x_{n_k}(t)-\tilde{x}(t)\Vert\leq y(t)$ almost everywhere on $[0,T]$, using the dominated convergence theorem~\cite[Theorem~3, p.~45]{diestel1977}, we can deduce that $\Vert x_{n_k}-\tilde{x}\Vert_1 \to 0$. Since $x_n$ converges strongly to $x$ in $L^1([0,T],{\mathcal H})$ and $L^1([0,T],{\mathcal H})$ is a subspace of $L^1([0,T],{\mathcal H})$,  $x_n$ converges strongly to $x$ in $L^1([0,T],{\mathcal H})$. By the uniqueness of the limit, we have $\tilde{x}=x$. Therefore, we have shown that $x_{n_k}(t)$ converges to $x(t)$ almost everywhere on $[0,T]$.
	
	The proof is complete.
\end{proof}

\begin{remark} {\rm In the formulation of Lemma~\ref{stronglp_pointwise}, one can replace $L^1([0,T],{\mathcal H})$ by any Bochner space $L^p(\Omega,X)$ with $1\leq p<\infty$. The proof remains the same, provided that one writes $L^p(\Omega,X)$ instead of $L^1([0,T],{\mathcal H})$ and $L^p([0,T],\mathbb{R})$ instead of $L^1([0,T],\mathbb{R})$.}
\end{remark}

For more details on Bochner integration, we refer to~\cite[p.~132]{yosida1980},~\cite[Chapter~II]{diestel1977},~\cite[Section~1.4]{cazenave_haraux_1998}, and~\cite[p.~116]{brezis2011}.

Now, we recall the definition and some properties of Sobolev spaces of vector-valued functions. Let $\Omega$ be an open subset of $\mathbb{R}$ and $X$ be a Banach space. The space $ L^1_{\rm loc}(\Omega,X)$ of \textit{locally integrable} functions is defined as follows:
$$L^1_{\rm loc}(\Omega,X) :=\left\{ f:\Omega\to X\mid \int_K \Vert f(\tau)\Vert d\tau <\infty, \ \forall K\subset\Omega, \ K \ \text{is\ compact}\right\}.$$

\begin{definition}
	Let $f\in L^p(\Omega,X)$, where $p\in [1,\infty)$, a function $\tilde{f}\in L^1_{\rm loc}(\Omega,X)$ is said to be a \textit{weak derivative} of $f$ if 
	$$\int_\Omega  \dot{g}(\tau) f(\tau)d\tau=-\int_\Omega g(\tau)\tilde{f}(\tau) d\tau,$$
	for all $g\in C^\infty_0(\Omega)$, where $C^\infty_0(\Omega)$ the space of all real-valued functions that are infinitely differentiable and have compact support in $\Omega$.
\end{definition}

The weak derivative of $f\in L^p(\Omega,X)$ is uniquely defined up to a set of measure zero (see~\cite[Proposition~23.18]{zeidler_2A_1990}). 

\begin{definition}\label{def_sobolev_space}\rm {(See, e.g,~\cite{cazenave_haraux_1998})
	Let $p\in[1,+\infty)$, $\Omega$ be an open subset of $\mathbb{R}$, and $X$ be a Banach space. The \textit{Sobolev space} $W^{1,p}(\Omega,X)$ is the set of all functions $f\in L^p(\Omega,X)$ that admit a weak derivative on $\Omega$ satisfying $\dot{f}\in L^p(\Omega,X)$. This space is equipped with the norm
 $$\Vert f \Vert_{W^{1,p}} = \left(\int_\Omega\Vert f\Vert^p d\mu \right)^\frac{1}{p}+\left(\int_\Omega\Vert \dot{f}\Vert^p d\mu \right)^\frac{1}{p}.$$}
\end{definition}

From the above definition, we see that if a sequence $\{f_k\}$ converges strongly to $f$ in $W^{1,p}(\Omega,X)$, then $f_k$ (resp., $\dot{f}_k$) converges strongly to $f$ (resp., $\dot{f}$) in $L^p(\Omega,X)$. It is well known \cite[Proposition~1.4.34]{cazenave_haraux_1998} that $W^{1,p}(\Omega,X)$ is a Banach space for all $p\in[1,+\infty)$.

\begin{proposition}\label{sobolev_w11}{\rm (See~\cite[Theorem~1.4.35]{cazenave_haraux_1998})}
	Let $p\in[1,\infty)$ and $x\in L^p(\Omega, X)$. The following conditions are equivalent
	\begin{itemize}
		\item[{\rm (a)}] $x\in W^{1,p}(\Omega,X)$.
		\item[{\rm (b)}] $x$ is absolutely continuous, differentiable almost everywhere and $\dot{x}\in L^p(\Omega, X)$.
		\item[{\rm (c)}] there exists a function $y\in L^p(\Omega, X)$ such that for almost every $t_0,t\in\Omega$, one has $$x(t)=x(t_0)+\displaystyle\int_{t_0}^t y(\tau)d\tau.$$
	\end{itemize}
\end{proposition}

\begin{remark} {\rm For $\Omega =(0,T)$, if $x:\Omega\to {\mathcal H}$ is an absolutely continuous function, then it is a simple matter to prove that the limits $\lim\limits_{t\to 0^+}x(t)$ and $\lim\limits_{t\to T^-}x(t)$ exist. So, setting $x(0)=\lim\limits_{t\to 0^+}x(t)$ and $x(T)=\lim\limits_{t\to T^-}x(t)$ gives an absolutely continuous function defined on $[0,T]$. Therefore, by Proposition~\ref{sobolev_w11} one can identify the Sobolev space $W^{1,1}(\Omega,X)$, where $\Omega =(0,T)$, with the space of absolutely continuous functions $u:[0,T]\to {\mathcal H}$ equipped with the norm
		\begin{equation}\label{norm_in_W1,1}\Vert u \Vert_{W^{1,1}} = \displaystyle\int_0^T\Vert u(\tau)\Vert d\tau +\displaystyle\int_0^T\Vert \dot{u}(\tau)\Vert d\tau.
		\end{equation} We use this identification and write $W^{1,1}([0,T],{\mathcal H})$ for $W^{1,1}((0,T),{\mathcal H})$.}
\end{remark}

 Throughout this paper, $A_0,A_1:{\mathcal H}\to {\mathcal H}$ are  positive semi-definite, bounded symmetric linear operators and $f:[0,T]\to {\mathcal H}$ is a continuous mapping. We denote by ${\rm Sol}(P,u_0)$ the solution set of~\eqref{main} with the initial value $u_0$. Before investigating the solution properties for problem~\eqref{main}, we present some assumptions that were used in preceding works~\cite{aht,ak18,ATY_2021}. 
\begin{assumption}{(H}{1)}
	\label{closedconvexassumption}
	\textit{The constraint sets $C(t)$, $t\in [0,T]$, are nonempty, closed, and convex. }
\end{assumption}

\begin{assumption}{(H}{1a)}
	\label{convexassumption}
	\textit{The constraint sets $C(t)$, $t\in [0,T]$, are nonempty and convex. }
\end{assumption}

\begin{assumption}{(H}{2a)}
	\label{continuousassumption}
	\textit{The set-valued mapping $C$ is continuous in the Hausdorff distance sense, i.e., there exists a continuous function $g:[0,T]\to \mathbb{R}$ such that
	\begin{equation}
	\label{continuousset}
	d_H(C(s),C(t))\leq   \vert g(s)-g(t)\vert\, \quad \forall s,t \in [0,T].
	\end{equation}}
\end{assumption}

\begin{assumption}{(H}{2b)}
	\label{lipschitz_like_assumption} 
\textit{$C$ is Lipschitz-like around every point in its graph.}
\end{assumption}

\begin{assumption}{(H}{3a)}
	\label{boundedness_assumption} 
	\textit{The constraint set $C(0)$ is bounded.}
\end{assumption}

\begin{assumption}{(H}{3b)}
	\label{coercive_assumption} 
	\textit{There exist positive constants $c_1, c_2$  such that
	\begin{equation*}
	\langle A_1x,x\rangle \geq c_1\Vert x\Vert^{2}-c_2,\ \;\forall x\in C(0).
	\end{equation*} }
\end{assumption}

\begin{assumption}{(H}{3c)}
	\label{coercive_assumption_1} 
	\textit{There exist positive constants $c_1, c_2$  such that
	\begin{equation*}
		\langle A_1x,x\rangle \geq c_1\Vert x\Vert^{2}-c_2,\ \;\forall t\in [0,T],\;\forall x\in C(t).
	\end{equation*} }
\end{assumption}

 First, we recall some solution existence and solution uniqueness results of~\eqref{main}.

\begin{theorem}\label{AHT_theorem} {\rm (See \cite[Theorems~5.1]{aht})} Suppose that~$\mathcal{H}$ is separable. If~{\rm\ref{closedconvexassumption}},~{\rm \ref{continuousassumption}}, and~{\rm \ref{boundedness_assumption}} are satisfied, then~\eqref{main} has at least one Lipschitz solution. 
\end{theorem} 

\begin{theorem}\label{AK_theorem} {\rm (See~\cite[Theorem 1]{ak18})} Suppose that~$\mathcal{H}$ is separable. If the assumptions~{\rm\ref{closedconvexassumption}},~{\rm \ref{continuousassumption}}, and~{\rm \ref{coercive_assumption}} are satisfied, then~\eqref{main} has at least one Lipschitz solution. 
\end{theorem} 

\begin{theorem}\label{ATY_theorem} {\rm (See~\cite[Theorem~3.3]{ATY_2021})} Let $A_0=0$, $A_1:\mathcal{H}\to \mathcal{H}$ be coercive. If the assumptions~{\rm\ref{closedconvexassumption}} and~{\rm \ref{lipschitz_like_assumption}} are satisfied, then~\eqref{main} has a unique solution $u$, which is a Lipschitz function. Moreover, the unique solution is a continuously differentiable function.
\end{theorem} 

In~\cite{aht,ATY_2021}, some conditions for the solution uniqueness of~\eqref{main}, which require the coerciveness of either $A_0$ or $A_1$, have been given.

\begin{theorem}\label{A_0_solution_uniqueness} {\rm (See~\cite[Theorem~5.2]{aht})}
	If $A_0$ is coercive and $C(t)$ is nonempty and convex for every $t\in [0,T]$,  then~\eqref{main} has at most one solution.
\end{theorem}
\begin{theorem}\label{A1_solution_uniqueness} {\rm (See~\cite[Theorem~3.4]{ATY_2021})}
	If $A_1$ is coercive and $C(t)$ is nonempty and convex for every $t\in [0,T]$, then~\eqref{main} has at most one solution.
\end{theorem}

For a detailed discussion on the above assumptions and results, we refer to~\cite{ATY_2021}.
\section{Solution Sensitivity with respect to the Initial Value}\label{section_SP_solutionsensitivity}
In this section, we investigate the solution sensitivity of~\eqref{main} with respect to the initial value when the solution is unique. The following theorem takes account of the case where the operator $A_0$ is coercive.
\begin{theorem}\label{continuity_dependence_1} If the assumption~{\rm \ref{convexassumption}} is satisfied, ${\rm Sol}(P,u_0)$ is nonempty for every $u_0\in C(0)$, and $A_0$ is coercive with the modulus of coercivity $\alpha_0$, then the mapping $\varphi:C(0)\to \mathcal{C}^0([0,T],{\mathcal H})$, $u_0 \mapsto u(u_0,\cdot)$, where $u(u_0,\cdot)$ denotes the unique solution of~\eqref{main}, is Lipschitz continuous with the modulus $\sqrt{\frac{\Vert A_0\Vert}{\alpha_0}}$. 
\end{theorem}
\begin{proof}
Let $x_0, y_0\in C(0)$ be given arbitrarily. Then, by our assumptions and Theorem~\ref{A_0_solution_uniqueness}, the sweeping process~\eqref{main} has a unique solution $x(\cdot)$ with the initial value $x_0$ (resp., a unique solution $y(\cdot)$ with the initial value $y_0$). Since $C(t)$ is convex, the inclusion \begin{equation}\label{basic_inclusion} A_1 \dot{u}(t)+A_0 u(t)-f(t) \in -{\mathcal N}_{C(t)}(\dot{u}(t))
\end{equation} in the formulation of \eqref{main} can be rewritten equivalently as $$\langle A_1 \dot{u}(t)+A_0 u(t)-f(t),  \dot{u}(t)- z\rangle \leq 0\quad \forall z\in C(t).$$ As ${\mathcal N}_{C(t)}(\dot{u}(t))=\emptyset$ if $\dot{u}(t)\notin C(t)$, the fulfillment of~\eqref{basic_inclusion} for almost every $t\in [0,T]$ implies that $\dot{u}(t)\in C(t)$ for almost every $t\in [0,T]$. Hence, the inclusions $\dot{x}(t)\in C(t)$ and $\dot{y}(t)\in C(t)$ hold for almost every $t\in [0,T]$. So, we have 
	\begin{equation}\label{inequalities}
	\begin{cases}
	\langle A_1 \dot{x}(t)+A_0 x(t)-f(t),  \dot{x}(t)- \dot{y}(t)\rangle \leq 0,\\
	\langle -A_1 \dot{y}(t)-A_0 y(t)+f(t),  \dot{x}(t)- \dot{y}(t)\rangle \leq 0
	\end{cases}
	\end{equation}	for almost every $t\in[0,T]$. Adding the inequalities in \eqref{inequalities} side by side yields 
$$\langle A_1 (\dot{x}(t)- \dot{y}(t)),\dot{x}(t)- \dot{y}(t)\rangle + \langle A_0 (x(t)- y(t)),\dot{x}(t)- \dot{y}(t)\rangle \leq 0\quad \text{a.e.} \; t \in [0,T].$$
	Since $A_1$ is positive semi-definite, this implies that
		\begin{equation}\label{xy_estimate}
		\langle A_0 (x(t)- y(t)),  \dot{x}(t)- \dot{y}(t)\rangle \leq 0\quad \text{a.e.} \; t \in [0,T].	
		\end{equation}
Taking the Lebesgue integral on both sides of the inequality in~\eqref{xy_estimate} and applying~\cite[Remarks~11.23(c)]{rudinprinciples}, we obtain
		\begin{equation}\label{xy_estimate_1}
	\int_0^t\langle A_0 (x(\tau)- y(\tau)),  \dot{x}(\tau)- \dot{y}(\tau)\rangle d\tau\leq 0\quad (\forall t\in [0,T]).	
	\end{equation}
	As $\dfrac{d}{d\tau}\langle A_0(x(\tau)-y(\tau)),x(\tau)-y(\tau)\rangle = 2\langle A_0(x(\tau)-y(\tau)), \dot{x}(\tau) - \dot{y}(\tau)\rangle$ at every point $\tau$ where both derivatives $\dot{x}(\tau), \dot{y}(\tau)$ exist, by~\cite[Theorem~6, p.~340]{kolmogorov1975} one has 
\begin{equation}\label{Newton_Lebnitz_formula}\begin{array}{rl}
	\displaystyle\int_0^t\langle A_0( x(\tau)-y(\tau)), \dot{x}(\tau) - \dot{y}(\tau)\rangle d\tau & =\dfrac{1}{2}\big[\langle A_0 (x(t)-y(t)), x(t)- y(t)\rangle\\  &\quad - \langle A_0 (x(0)-y(0)), x(0)- y(0)\rangle \big].
	\end{array}\end{equation}
	Then, from~\eqref{xy_estimate_1} it follows that $\langle A_0 (x(t)- y(t)),  x(t)- y(t)\rangle-\langle A_0 (x_0-y_0), x_0- y_0\rangle\leq 0$.
	Hence, by the coerciveness of $A_0$, we get
	\begin{align*}
	\alpha_0\Vert x(t)- y(t)\Vert ^2 \leq\langle A_0 (x(t)-y(t)),  x(t)- y(t)\rangle&\leq \langle A_0 (x_0-y_0),  x_0- y_0\rangle&\\\leq \Vert A_0 \Vert \Vert x_0-y_0\Vert^2.
	\end{align*}
	Therefore, $\Vert x(t)- y(t)\Vert\leq \sqrt{\frac{\Vert A_0\Vert}{\alpha_0}}\Vert x_0-y_0\Vert$ for all  $t\in[0,T]$. So, the inequality $$\Vert x-y\Vert_{\mathcal{C}^0}\leq \sqrt{\frac{\Vert A_0\Vert}{\alpha_0}}\Vert x_0-y_0\Vert$$ holds for any $x_0,y_0\in C(0)$. We have thus proved that the mapping $\varphi$ is Lipschitz continuous on $C(0)$ with the modulus $\sqrt{\frac{\Vert A_0\Vert}{\alpha_0}}$.
\end{proof}

According to Theorem~\ref{A1_solution_uniqueness}, the nonemptiness and convexity of $C(t)$ together with the coerciveness of $A_1$ can also guarantee the solution uniqueness for~\eqref{main} if such a solution exists. A natural question arises: \textit{Could we get a similar result as the one in Theorem~\ref{continuity_dependence_1} for the case under consideration?} The next theorem gives a complete answer to this question.

\begin{theorem}\label{continuity_dependence_2}
	If the assumption~{\rm \ref{convexassumption}} is fulfilled, ${\rm Sol}(P,u_0)$ is nonempty for every $u_0\in C(0)$, and $A_1$ is coercive with the modulus of coercivity $\alpha_1$, then the mapping $\varphi:C(0)\to \mathcal{C}^0([0,T],{\mathcal H})$, $u_0 \mapsto u(u_0,\cdot)$, where $u(u_0,\cdot)$ denotes the unique solution of~\eqref{main}, is Lipschitz continuous with the modulus $\sqrt{\dfrac{T\Vert A_0\Vert}{2\alpha_1}}+1$. 
\end{theorem}

\begin{proof}
For any $x_0, y_0\in C(0)$, the assumptions made and Theorem~\ref{A1_solution_uniqueness} assure that~\eqref{main} has a unique solution $x(\cdot)$ (resp., $y(\cdot)$) with the initial value $x_0$ (resp., $y_0$).  Then, arguing similarly as in the proof of Theorem~\ref{continuity_dependence_1}, we have
$$\langle A_1 \dot{x}(t)+A_0 x(t)-f(t), \dot{x}(t) - \dot{y}(t)\rangle \leq 0$$
and
$$\langle A_1 \dot{y}(t)+A_0 y(t)-f(t), \dot{y}(t) - \dot{x}(t) \rangle \leq 0$$ for almost every $t\in[0,T]$. Adding the last inequalities side by side, one obtains
\begin{equation}\label{xy_estimate2}
\langle A_1 (\dot{x}(t)-\dot{y}(t)),\dot{x}(t) - \dot{y}(t)\rangle+\langle A_0(x(t)- y(t)), \dot{x}(t) - \dot{y}(t)\rangle \leq 0
\end{equation} for almost every $t\in[0,T]$. Combining the coerciveness of $A_0$ with~\eqref{xy_estimate2} yields
\begin{equation}\label{xy_estimate3}
\alpha_1\Vert \dot{x}(t)-\dot{y}(t)\Vert^2\leq -\langle A_0 (x(t)-y(t)), \dot{x}(t) - \dot{y}(t)\rangle \quad \text{a.e.} \; t \in [0,T].
\end{equation} Since the function $t\mapsto -\langle A_0 (x(t)-y(t)), \dot{x}(t) - \dot{y}(t)\rangle$ is integrable (in the Lebesgue sense), from~\eqref{xy_estimate3} we can deduce that the function $t\mapsto \alpha_1\Vert \dot{x}(t)-\dot{y}(t)\Vert^2$ is also integrable.
Integrating both sides of the inequality in~\eqref{xy_estimate3}, we obtain
\begin{equation}\label{xy_estimate4}\displaystyle\int_0^t\alpha_1\Vert \dot{x}(\tau)-\dot{y}(\tau)\Vert^2d\tau\leq -\displaystyle\int_0^t\langle A_0( x(\tau)-y(\tau)), \dot{x}(\tau) - \dot{y}(\tau)\rangle d\tau.	
\end{equation}
At every point $\tau$ where both derivatives $\dot{x}(\tau), \dot{y}(\tau)$ exist, we have $$\dfrac{d}{d\tau}\langle A_0(x(\tau)-y(\tau)),x(\tau)-y(\tau)\rangle = 2\langle A_0(x(\tau)-y(\tau)), \dot{x}(\tau) - \dot{y}(\tau)\rangle.$$ Hence, as noted in the preceding proof, by~\cite[Theorem~6, p.~340]{kolmogorov1975} we have \eqref{Newton_Lebnitz_formula}. Consequently, from~\eqref{xy_estimate4} it follows that
$$\displaystyle\int_0^t\alpha_1\Vert \dot{x}(\tau)-\dot{y}(\tau)\Vert^2d\tau\leq -\dfrac{1}{2}\big[\langle A_0 (x(t)-y(t)), x(t)- y(t)\rangle- \langle A_0 (x(0)-y(0)), x(0)- y(0)\rangle \big].$$
Since $A_0$ is positive semidefinite, the latter implies
$$\displaystyle\int_0^t\alpha_1\Vert \dot{x}(\tau)-\dot{y}(\tau)\Vert^2d\tau\leq \dfrac{1}{2} \langle A_0 (x(0)-y(0)), x(0)- y(0)\rangle \leq \dfrac{\Vert A_0\Vert}{2}\Vert x_0-y_0\Vert^2.$$
So, we have
\begin{equation}\label{xy_estimate5}
	\displaystyle\int_0^t\Vert \dot{x}(\tau)-\dot{y}(\tau)\Vert^2d\tau\leq \dfrac{\Vert A_0\Vert}{2\alpha_1}\Vert x_0-y_0\Vert^2.
\end{equation}
In addition, for each $t\in[0,T]$ one has
\begin{equation}\label{xy_estimate6}
	\begin{array}{rl}
	\Vert x(t)-y(t)\Vert &= \Big\Vert \left(x_0+	\displaystyle\int_0^t \dot{x}(\tau)d\tau \right)-\left( y_0+\displaystyle\int_0^t \dot{y}(\tau)d\tau \right)\Big\Vert\\
	&\leq \Vert x_0-y_0\Vert + 	\displaystyle\int_0^t\Vert \dot{x}(\tau)-\dot{y}(\tau)\Vert d\tau.
	\end{array}
\end{equation}
The inequality shows that the function $t\mapsto \Vert \dot{x}(t)-\dot{y}(t)\Vert$ belongs to the space $L^2([0,T],\mathbb R)$. Therefore, setting $\beta(t)=1$ for $t\in [0,T]$ and using the H\"older's inequality (see~\cite[Theorem~4.6]{brezis2011} and~\cite[p.~385]{kolmogorov1975}) for functions from $L^2([0,T],\mathbb R)$, we have $$\displaystyle\int_0^t(\beta(\tau)\Vert \dot{x}(\tau)-\dot{y}(\tau)\Vert) d\tau\leq \left(\displaystyle\int_0^t\beta(\tau)^2 d\tau\right)^\frac{1}{2}\left(\displaystyle\int_0^t\Vert \dot{x}(\tau)-\dot{y}(\tau)\Vert^2 d\tau\right)^\frac{1}{2}.$$
Then, combining this with~\eqref{xy_estimate5} yields
$$\displaystyle\int_0^t\Vert \dot{x}(\tau)-\dot{y}(\tau)\Vert d\tau\leq \sqrt{t} \sqrt{\dfrac{\Vert A_0\Vert}{2\alpha_1}}\Vert x_0-y_0\Vert \leq \sqrt{T} \sqrt{\dfrac{\Vert A_0\Vert}{2\alpha_1}}\Vert x_0-y_0\Vert$$ for every $t\in[0,T]$. Hence, thanks to~\eqref{xy_estimate6}, we get 
$$\Vert x(t)-y(t)\Vert\leq  \Vert x_0-y_0\Vert+\sqrt{\dfrac{T\Vert A_0\Vert}{2\alpha_1}}\Vert x_0-y_0\Vert = \left(\sqrt{\dfrac{T\Vert A_0\Vert}{2\alpha_1}}+1\right)\Vert x_0-y_0\Vert $$
for all $t\in[0,T]$. This implies that the mapping $\varphi$ defined in the statement of the theorem is Lipschitz continuous on $C(0)$ with the modulus $\sqrt{\dfrac{T\Vert A_0\Vert}{2\alpha_1}}+1$.
\end{proof}

\section{Boundedness of the Solution Set}\label{section_SP_boundedness}

Noting that the Sobolev space $W^{1,1}([0,T],{\mathcal H})$ is the space of all absolutely continuous functions with its derivative in $L^1([0,T],{\mathcal H})$ (see Proposition~\ref{sobolev_w11}), we can view the solution set of~\eqref{main} as a subset of $W^{1,1}([0,T],{\mathcal H})$. Of course, at the same time, it is a subset of $\mathcal{C}^0([0,T],{\mathcal H})$.

If~\eqref{main} has a unique solution then, under suitable conditions, we have established the solution sensitivity with respect to the initial value. When the solution uniqueness is not guaranteed, the solution set of~\eqref{main} may be unbounded. Let us consider an example.

\begin{example}{\rm Let ${\mathcal H}=\mathbb{R}^2$, $A_0=A_1=\begin{pmatrix} 0 & 0\\ 0 & 1 \end{pmatrix}$, $u_0=(0,0)$, $f(t) =(0,t)$, and $C(t)=\mathbb{R}\times\{0\}$ for all $t\in[0,T]$. For every $\lambda\in\mathbb{R}$, we define a function by setting $u^{(\lambda)}(t)=(\lambda t,0)$ for all $t\in [0,T]$. Clearly, $u^{(\lambda)}(0)=(0,0)$ and $\dot{u}^{(\lambda)}(t) =(\lambda,0)\in C(t)$ for all $t\in[0,T]$. In addition,
		\begin{equation*}
		A_1 \dot{u}^{(\lambda)}(t)+A_0 u^{(\lambda)}(t)-f(t)=\begin{pmatrix} 0 & 0\\ 0 & 1 \end{pmatrix} \begin{pmatrix} \dot{u}^{(\lambda)}_1(t)\\ \dot{u}^{(\lambda)}_2(t) \end{pmatrix}+\begin{pmatrix} 0 & 0\\ 0 & 1 \end{pmatrix} \begin{pmatrix} u^{(\lambda)}_1(t)\\ u^{(\lambda)}_2(t) \end{pmatrix}-\begin{pmatrix} 0\\ t\end{pmatrix} = \begin{pmatrix} 0\\ -t\end{pmatrix}. 
		\end{equation*} 
Since $\mathcal{N}_{C(t)}(\dot{u}^{(\lambda)}(t))=\{0\}\times\mathbb{R}$, this yields $A_1 \dot{u}^{(\lambda)}(t)+A_0 u^{(\lambda)}(t)-f(t)\in-\mathcal{N}_{C(t)}(\dot{u}^{(\lambda)}(t))$ for all $t\in[0,T]$. Thus, for any $\lambda\in\mathbb{R}$, $u^{(\lambda)}$ is a solution of~\eqref{main}. As $\Vert u^{(\lambda)}\Vert_{\mathcal{C}^0} = \vert \lambda \vert T$, the solutions of~\eqref{main} form an unbounded subset of $\mathcal{C}^0([0,T],{\mathcal H})$.
}
\end{example}

Our aim in this section is to establish some sets of conditions ensuring that the solution set of~\eqref{main} is bounded.

\begin{theorem}\label{boundedness_sol_1}
	If $C(t)$ is nonempty for all $t\in[0,T]$ and the assumptions~{\rm \ref{continuousassumption}}, {\rm \ref{boundedness_assumption}} are satisfied then,  for any $u_0\in C(0)$, the solution set ${\rm Sol}(P,u_0)$ is bounded in both spaces  $\mathcal{C}^0([0,T],{\mathcal H})$ and $W^{1,1}([0,T],{\mathcal H})$.
\end{theorem}
\begin{proof} Let $u_0\in C(0)$ be given arbitrarily. If ${\rm Sol}(P,u_0)$ is empty, then it is bounded. Suppose that ${\rm Sol}(P,u_0)\neq\emptyset$ and $u$ is an element from ${\rm Sol}(P,u_0)$. As $C(0)$ is bounded, we can find $\rho_0>0$ such that $C(0)\subset \rho_0\bar{\mathbb{B}}(0,1)$. Let $g:[0,T]\to\mathbb{R}$ be  a continuous function satisfying~\eqref{continuousset}. Thus, for all $t\in[0,T]$ one has $C(t)\subset C(0)+\vert g(0)-g(t)\vert$. So, $C(t)\subset \rho\bar{\mathbb{B}}(0,1)$ for all $t\in[0,T]$, where $\rho:=\rho_0+\max\{\vert g(0)-g(s)\vert\mid s\in[0,T]\}$. Since $\dot{u}(t)\in C(t)$ for almost every $t\in[0,T]$, one has $\Vert \dot{u}(t)\Vert \leq \rho$ for almost every $t\in[0,T]$. For any $t\in[0,T]$, we define two sets $\Omega_1(t)=\{s\in[0,t]\mid \Vert \dot{u}(s)\Vert \leq \rho \}$ and $\Omega_2(t)=\{s\in[0,t]\mid \Vert \dot{u}(s)\Vert > \rho \}$. Then, the sets $\Omega_1(t)$ and $\Omega_2(t)$ are measurable, and $\mu(\Omega_2(t))=0$ with $\mu$ being the Lebesgue measure on $\mathbb{R}$. So, by \cite[Remark~3.4(c)]{ATY_2021} and \cite[Theorem~4, p.~46]{diestel1977}, we have
	$$\begin{array}{rl}
		\|u(t)\|=\|u_0+\displaystyle\int_0^{t}\dot{u}(\tau)d\tau\|&=\|u_0+\displaystyle\int_{\Omega_1(t)} \dot{u}(\tau)d\tau+\displaystyle\int_{\Omega_2(t)} \dot{u}(\tau)d\tau\|\\&\leq \Vert u_0\Vert +\displaystyle\int_{\Omega_1(t)} \Vert\dot{u}(\tau)\Vert d\tau+\displaystyle\int_{\Omega_2(t)} \Vert\dot{u}(\tau)\Vert d\tau\\
		&\leq \Vert u_0\Vert + \rho\mu(\Omega_1(t))\\
		&\leq \Vert u_0 \Vert + \rho T.
	\end{array}$$
 Thus, $\Vert u \Vert_{\mathcal{C}^0} \leq\Vert u_0 \Vert + \rho T$. This establishes the boundedness of ${\rm Sol}(P,u_0)$ in $\mathcal{C}^0([0,T],{\mathcal H})$.	Since
   $\|u(t)\|\leq \Vert u_0 \Vert + \rho T$ for all $t\in[0,T]$, $\Vert \dot{u}(t)\Vert \leq \rho$ for a.e. $t\in[0,T]$, and $u\in {\rm Sol}(P,u_0)$ was chosen arbitrarily, by~\eqref{norm_in_W1,1} we can assert that ${\rm Sol}(P,u_0)$ is a bounded subset of the Sobolev space $W^{1,1}([0,T],{\mathcal H})$.
\end{proof}

To deal with the case where the sets $C(t)$, $t\in [0,T]$, can be unbounded, we will need the following technical lemma. Since we still have not found any reference containing this statement, a detailed proof is given here.  

\begin{lemma}\label{gronwall_type}
	Let $f$ be a Lebesgue integrable, real-valued function defined on $[0,T]$. If 
	\begin{equation}\label{gronwall_type_inq}
	f(t)\leq a+b\displaystyle\int_0^t f(\tau)d\tau \quad {\rm a.e.}\ t\in[0,T] 
	\end{equation}
	for some constants $a,b$ with $b\neq 0$, then $\displaystyle\int_0^t f(\tau)d\tau  \leq \dfrac{a}{b}(\exp(bt)-1)$ for all $t\in[0,T]$.
\end{lemma}
\begin{proof}
	Let $f$ be a Lebesgue integrable function on $[0,T]$ satisfying~\eqref{gronwall_type_inq}. Multiplying both sides of the inequality in~\eqref{gronwall_type_inq} by $\exp(-bt)$ yields
	\begin{equation}\label{gronwall_type_inq1}
	\exp(-bt)f(t) -b\exp(-bt)\displaystyle\int_0^t f(\tau)d\tau \leq a\exp(-bt) \quad {\rm a.e.}\ t\in[0,T].
	\end{equation} By~\cite[Theorem~8, p.~324]{kolmogorov1975}, one has 
	$$\dfrac{d}{ds} \left(\exp(-bs)\displaystyle\int_0^s f(\tau)d\tau \right)= \exp(-bs)f(s)-b\exp(-bs)\displaystyle\int_0^s f(\tau)d\tau.$$
	Thus, taking the Lebesgue integral on both sides of the inequality in~\eqref{gronwall_type_inq1} and applying~\cite[Remarks~11.23(c)]{rudinprinciples}, we obtain
	$$\displaystyle\int_0^t \dfrac{d}{ds} \left(\exp(-bs)\displaystyle\int_0^s f(\tau)d\tau \right)ds \leq \displaystyle\int_0^t a\exp(-bs) ds \quad \forall t\in[0,T].$$
	It follows that
	$$\exp(-bt) \displaystyle\int_0^t f(\tau)d\tau  \leq \dfrac{a}{b} (1- \exp(-bt))\quad \forall t\in[0,T].$$
	Hence, we get
	$$ \displaystyle\int_0^t f(\tau)d\tau  \leq \dfrac{a}{b}(\exp(bt)-1)\quad \forall t\in[0,T].$$
	
	The proof is complete.
\end{proof}
\begin{theorem}\label{boundedness_sol_2}
	If the assumptions~{\rm\ref{convexassumption}},~{\rm \ref{continuousassumption}} and {\rm \ref{coercive_assumption}} are satisfied then, for any $u_0\in C(0)$, the solution set ${\rm Sol}(P,u_0)$ is bounded in both spaces  $\mathcal{C}^0([0,T],{\mathcal H})$ and $W^{1,1}([0,T],{\mathcal H})$.
\end{theorem}
\begin{proof} Given any $u_0\in C(0)$. If ${\rm Sol}(P,u_0)$ is empty, then it is bounded. Suppose that ${\rm Sol}(P,u_0)$ is nonempty. Take any $u\in {\rm Sol}(P,u_0)$ and let $\varepsilon>0$ be given arbitrarily. Since $C(t) $ is nonempty, for any $t\in[0,T]$ there exists $z_t\in C(t)$ satisfying $$\Vert u_0-z_t\Vert < d(u_0, C(t))+\varepsilon.$$ By \ref{continuousassumption}, we have
\begin{align*} 
	\|z_t\|-\|u_0\|\leq\Vert u_0-z_t\Vert < d(u_0, C(t))+\varepsilon&\leq d_H(C(0), C(t)) +\varepsilon\\
	&\leq \vert g(0)-g(t)\vert +\varepsilon.
	\end{align*}
Then, setting $\beta:= \|u_0\|+\max\limits_{\tau\in [0,T]}\vert g(0)-g(\tau)\vert +\varepsilon$, we get $\Vert z_t\Vert < \beta$. So, for every $t\in [0, T]$ one can find some $z_t\in C(t)$ such that $\Vert z_t\Vert < \beta$. As $u\in {\rm Sol}(P,u_0)$, by~{\rm\ref{convexassumption}} one has for almost every $t\in[0,T]$ that
$$\langle A_1 \dot{u}(t)+A_0 u(t)-f(t),  \dot{u}(t)- z\rangle \leq 0\quad \forall z\in C(t).$$
Substituting $z=z_t$ into the above inequality yields $$\langle A_1 \dot{u}(t)+A_0 u(t)-f(t),  \dot{u}(t)- z_t\rangle \leq 0$$
for almost every $t\in[0,T]$. Thus,
\begin{equation}\label{dotu_estimate}
	\langle A_1 \dot{u}(t),\dot{u}(t)\rangle -\langle A_1 \dot{u}(t),z_t\rangle +\langle A_0 u(t)-f(t),  \dot{u}(t)\rangle -\langle A_0 u(t)-f(t), z_t\rangle  \leq 0.
\end{equation}
Using the assumptions~{\rm \ref{continuousassumption}},~{\rm \ref{coercive_assumption}}, and~\cite[Remark~3.2]{ATY_2021}, we can find positive constants $\hat c_1, \hat c_2$  such that
$\langle A_1x,x\rangle \geq \hat c_1\Vert x\Vert^{2}-\hat c_2$ for all $t\in [0,T]$ and $x\in C(t)$. Then,~\eqref{dotu_estimate} implies that
$$\hat c_1\Vert \dot{u}(t)\Vert^{2}-\hat c_2-\langle A_1 \dot{u}(t),z_t\rangle +\langle A_0 u(t)-f(t),  \dot{u}(t)\rangle -\langle A_0 u(t)-f(t), z_t\rangle  \leq 0$$
for a.e. $t\in[0,T]$. So, one has
$$\hat c_1\Vert \dot{u}(t)\Vert^{2}-\hat c_2- \beta\Vert A_1\Vert \Vert\dot{u}(t)\Vert -(\Vert A_0\Vert \Vert u(t)\Vert +\Vert f\Vert_{\mathcal{C}^0}) \Vert \dot{u}(t)\Vert -\beta(\Vert A_0 \Vert \Vert u(t)\Vert  +\Vert f\Vert_{\mathcal{C}^0}) \leq 0$$
for a.e. $t\in[0,T]$. For each $t\in[0,T]$, setting $a_1(t)= \beta\Vert A_1\Vert +\Vert A_0\Vert \Vert u(t)\Vert +\Vert f\Vert_{\mathcal{C}^0}$ and $$a_2(t)= \beta(\Vert A_0 \Vert \Vert u(t)\Vert  +\Vert f\Vert_{\mathcal{C}^0}) + \hat c_2,$$ we get 
\begin{equation}\label{dotu_estimate1}
	\hat c_1\Vert \dot{u}(t)\Vert^{2} -a_1(t)\Vert\dot{u}(t)\Vert-a_2(t) \leq 0\quad {\rm a.e.}\ t\in[0,T].
\end{equation} As one has $\hat c_1 >0$ and $a_2(t)>0$ for every $t\in [0,T]$, the quadratic polynomial $q(x) := \hat c_1 x^2- a_1(t) x -a_2(t)$ has two roots with different signs.  Hence,~\eqref{dotu_estimate1} holds if and only if 
$$\Vert \dot{u}(t)\Vert \leq \dfrac{a_1(t)+\sqrt{a_1(t)^2-4\hat{c}_1a_2(t)}}{2\hat{c}_1}\quad\ {\rm a.e.}\ t\in[0,T].$$ Since $\sqrt{a_1(t)^2-4\hat{c}_1a_2(t)} \leq a_1(t)$, this yields $\Vert \dot{u}(t)\Vert \leq \dfrac{a_1(t)}{\hat{c}_1}$ for a.e. $t\in[0,T]$. Therefore, $$\Vert \dot{u}(t)\Vert \leq \dfrac{\beta\Vert A_1\Vert +\Vert A_0\Vert \Vert u(t)\Vert +\Vert f\Vert_{\mathcal{C}^0}}{\hat{c}_1}$$ for a.e. $t\in[0,T]$. Then one has 
\begin{equation}\label{estimate_for_dot_u}\Vert \dot{u}(t)\Vert \leq \gamma (1+\Vert u(t) \Vert)\quad {\rm a.e.}\ t\in[0,T],
\end{equation}
where $\gamma := \max\left\{\dfrac{\beta\Vert A_1\Vert +\Vert f\Vert_{\mathcal{C}^0}}{\hat{c}_1},\dfrac{\Vert A_0\Vert}{\hat{c}_1}\right\}$. Since
\begin{equation}\label{u_estimate}
	\|u(t)\|=\|u_0+\displaystyle\int_0^{t}\dot{u}(\tau)d\tau\|\leq \Vert u_0 \Vert+\displaystyle\int_0^t \Vert \dot{u}(\tau)\Vert d\tau
\end{equation}
(see \cite[Remark~3.4(c)]{ATY_2021} and \cite[Theorem~4(ii), p.~46]{diestel1977}), from~\eqref{estimate_for_dot_u} it follows that $$\Vert \dot{u}(t)\Vert \leq \gamma (1+ \Vert u_0 \Vert)+\gamma\displaystyle\int_0^t \Vert \dot{u}(\tau)\Vert d\tau\quad {\rm a.e.}\ t\in[0,T].$$ So, applying Lemma~\ref{gronwall_type} for $f(t):= \Vert \dot{u}(t)\Vert$, $a:= \gamma (1+ \Vert u_0 \Vert)$, and $b:=\gamma$ gives
$$\displaystyle\int_0^t \Vert \dot{u}(\tau)\Vert d\tau \leq (1+ \Vert u_0 \Vert) (\exp(\gamma t)-1) \leq  (1+ \Vert u_0 \Vert) (\exp(\gamma T)-1)\quad \forall t\in[0,T].$$
Combining this with~\eqref{u_estimate} yields 
\begin{equation}\label{estimate_for_u} \|u(t)\|\leq\|u_0\|+(1+ \Vert u_0 \Vert) (\exp(\gamma T)-1)\quad \forall t\in[0,T].
\end{equation}
It follows that $\|u\|_{\mathcal{C}^0}\leq\|u_0\|+(1+ \Vert u_0 \Vert) (\exp(\gamma T)-1)$. So, ${\rm Sol}(P,u_0)$ is a bounded subset of $\mathcal{C}^0([0,T],{\mathcal H})$. Finally, using the estimates~\eqref{estimate_for_dot_u},~\eqref{estimate_for_u}, and formula \eqref{norm_in_W1,1}, we can find a constant $\rho>0$ such that $\Vert u \Vert_{W^{1,1}}\leq\rho$ for any $u\in{\rm Sol}(P,u_0)$. The proof is complete.
\end{proof}

\begin{theorem}\label{boundedness_sol_3}
	If the assumptions~{\rm\ref{convexassumption}},~{\rm \ref{lipschitz_like_assumption}} and {\rm \ref{coercive_assumption_1}} are satisfied then, for any $u_0\in C(0)$, the solution set ${\rm Sol}(P,u_0)$ is bounded in both spaces  $\mathcal{C}^0([0,T],{\mathcal H})$ and $W^{1,1}([0,T],{\mathcal H})$.
\end{theorem}
\begin{proof}
For each $t\in[0,T]$, pick a point $x_t\in C(t)$. As $C$ is Lipschitz-like around $(t,x_t)$, there exist an open neighborhood $V_t$ of $t$ in the induced topology of $[0,T]\subset\mathbb R$, a neighborhood $W_t$ of $x_t$ in $\mathcal{H}$, and a constant $\kappa_t >0$ such that 
	\begin{equation}\label{lipschitz-like_inq_1}
		C(t')\cap W_t \subset C(t'') +\kappa_t\vert t'-t''\vert \bar{\mathbb{B}}(0,1)\quad\forall t',t''\in V_t.
	\end{equation}
	Since $[0,T]=\displaystyle\bigcup_{t\in[0,T]}V_t$, the compactness of $[0,T]$ implies the existence of $t_1,\dots,t_k$ in $[0,T]$ such that $[0,T]=\bigcup\limits_{i=1}^kV_{t_i}$. For each $i\in\{1,\dots,k\}$, we have $x_{t_i}\in  W_{t_i}$. So, thanks to~\eqref{lipschitz-like_inq_1}, for every $t\in V_{t_i}$ we can find $z^{(i)}_t\in C(t)$ and $\xi^{(i)}_t\in  \bar{\mathbb{B}}(0,1)$ satisfying $x_{t_i} =z^{(i)}_t+ \kappa_{t_i} \vert t-t_i\vert \xi^{(i)}_t$. Then, \begin{equation}\label{estimate_z_t_i}\Vert z^{(i)}_t\Vert \leq \Vert x_{t_i} \Vert +\kappa_{t_i} \vert t-t_i\vert\leq \Vert x_{t_i} \Vert +\kappa_{t_i} T.
	\end{equation} Setting $\beta=\max\big\{\Vert x_{t_i} \Vert +\kappa_{t_i} T\mid i\in\{1,\dots,k\}\big\}$, we have $\beta >0$. For each $t\in[0,T]$, there is some $i\in\{1,\dots,k\}$ such that $t\in V_{t_i}$ and, by~\eqref{estimate_z_t_i}, the element $z^{(i)}_t\in C(t)$ satisfies the estimate $\Vert z^{(i)}_t\Vert \leq \beta$. Therefore, for every $t\in[0,T]$, there exists at least one point of the form $z^{(i)}_t$ such that $z^{(i)}_t\in C(t)$ and $\Vert z^{(i)}_t\Vert \leq \beta$. 
	
	Let $u_0\in C(0)$ be given arbitrarily. Since ${\rm Sol}(P,u_0)$ bounded if it is empty, it suffices to consider the case  ${\rm Sol}(P,u_0)\neq\emptyset$. Take any $u\in{\rm Sol}(P,u_0)$. By~{\rm\ref{convexassumption}} we deduce for almost every $t\in[0,T]$ that
	$\langle A_1 \dot{u}(t)+A_0 u(t)-f(t),  \dot{u}(t)- z\rangle \leq 0$ for all $z\in C(t).$ Substituting $z=z^{(i)}_t$ into the last inequality yields $$\langle A_1 \dot{u}(t)+A_0 u(t)-f(t),  \dot{u}(t)- z^{(i)}_t\rangle \leq 0$$
	for almost every $t\in[0,T]$. Using the assumption~{\rm\ref{coercive_assumption_1}} and repeating the final part of the proof of Theorem~\ref{boundedness_sol_1} (starting from inequality~\eqref{dotu_estimate}), we can show that the solution set ${\rm Sol}(P,u_0)$ is bounded in both spaces $\mathcal{C}^0([0,T],{\mathcal H})$ and $W^{1,1}([0,T],{\mathcal H})$.
\end{proof}

\begin{remark}{\rm The boundedness of ${\rm Sol}(P,u_0)$ in Theorem~\ref{boundedness_sol_3} is also valid if instead of the assumption~{\rm \ref{lipschitz_like_assumption}} one requires that 	$C$ is inner semicontinuous at every point in its graph, i.e., for every $(t,x)\in [0,T]\times {\mathcal H}$ with $x\in C(t)$, if $U\subset {\mathcal H}$ is an open set containing $x$, then there exists a neighborhood $V$ of $t$ in $[0,T]$ such that $C(t')\cap U\neq\emptyset$ for all $t'\in V$. Indeed, for each $t\in[0,T]$, select a point $x_t\in C(t)$. The inner semicontinuity of $C$ at $(t,x_t)$ assures that there is an open neighborhood $V_t$ of $t$ in the induced topology of $[0,T]$ such that $C(t')\cap \mathbb{B}(x_t,1)\neq \emptyset$ for every $t'\in V_t$. By the compactness of $[0,T]$, from the open covering  $\{V_t\}_{t\in[0,T]}$ of the segment we can extract a finite subcover $V_{t_1},\dots,V_{t_k}$. So, for each $t\in[0,T]$, there exists an index $i\in\{1,\dots,k\}$ such that $t\in V_{t_i}$. Since $C(t)\cap \mathbb{B}(x_{t_i},1)\neq \emptyset$, there is a vector $z^{(i)}_t\in C(t)\cap \mathbb{B}(x_{t_i},1)$. Then one has $\Vert z^{(i)}_t\Vert\leq\beta$, where $\beta:= \max\big\{\Vert x_i\Vert + 1\mid i\in\{1,\dots,k\}\big\}$. Consequently, for each $t\in[0,T]$, there exists at least one point of the form $z^{(i)}_t$ such that $z^{(i)}_t\in C(t)$ and $\Vert z^{(i)}_t\Vert \leq \beta$. Then, as noted above, the usage of~{\rm\ref{coercive_assumption_1}} and the repetition of the final part of the proof of Theorem~\ref{boundedness_sol_1} yield the desired assertion.
}
\end{remark}

\begin{remark}{\rm
If a set-valued mapping is Lipschitz-like around a point in its graph then it is inner semicontinuous at that point (see, e.g.,~\cite[Proposition~3.1]{Yen_AMV87}). On the other hand, there exist locally Lipschitz-like mappings which are not continuous in the  Hausdorff distance sense (see~\cite[Example~3.1]{ATY_2021} and the discussion therein). Clearly, if the mapping $C:[0,T]\rightrightarrows {\mathcal H}$ is continuous in the  Hausdorff distance sense, then it is  inner semicontinuous at every point in its graph. The just cited example of~\cite{ATY_2021} shows that the converse is not true in general.
}
\end{remark}
\begin{remark}{\rm
		The continuity in the Hausdorff distance sense of  $C(\cdot)$ together with the assumption~{\rm\ref{coercive_assumption}} implies~{\rm \ref{coercive_assumption_1}} (see~\cite[Remark~3.2]{ATY_2021}). However, a similar implication may not hold under the inner semicontinuity of $C(\cdot)$ at every point in its graph or even under the Lipschitz-likeness of $C(\cdot)$ around every point in its graph.  
	}
\end{remark}

\section{Closedness of the Solution Set}\label{section_SP_closedness} 

First, let us show that the closedness of ${\rm Sol}(P,u_0)$ may not available even for very simple problems in finite dimensions.

\begin{proposition}
	The solution set of~\eqref{main} may not be closed in $\mathcal{C}^0([0,T],{\mathcal H})$.
\end{proposition}
\begin{proof}
	We will prove the proposition by constructing a suitable example. Let ${\mathcal H}=\mathbb{R}$, $A_0=0$, $A_1=0$, $u_0=0$, $f(t)\equiv 0$,  and $C(t)= \mathbb{R}$ for all $t\in [0,T]$. Then, an absolutely continuous function $u:[0,T]\to\mathbb{R}$ is a solution of~\eqref{main} if and only if
	\begin{equation*}
		\left\{
		\begin{array}{l}
			0 \in {\mathcal N}_{C(t)}(\dot{u}(t))\quad \text{a.e.} \; t \in [0,T],\\ 
			u(0)=0.
		\end{array}\right. 
	\end{equation*}
	Since $C(t)=\mathbb{R}$ for all $t\in [0,T]$, ${\mathcal N}_{C(t)}(\dot{u}(t)) =\{0\}$ for any $t$ where $\dot{u}(t)$ exists. So, any absolutely continuous function $u:[0,T]\to\mathbb{R}$ with $u(0)=0$ is a solution of~\eqref{main}. For $k\in\mathbb{N}$, let
	\begin{equation*}
		x_k(t)=\begin{cases}
			t^2 \sin(\frac{1}{t^2}) \quad &\text{if} \; t \in (\frac{1}{k},T]\\
			\frac{t}{k}\sin(k^2)&\text{if}\; t\in [0,\frac{1}{k}].
		\end{cases}
	\end{equation*}
	and
	\begin{equation*}
		x(t)=\begin{cases}
			t^2 \sin(\frac{1}{t^2}) \quad &\text{if} \; t \in (0,T]\\
			0&\text{if}\; t=0.
		\end{cases}
	\end{equation*}
	Clearly, $x_k(\cdot)$ is a Lipschitz function for each $k\in\mathbb{N}$. Since $x_k(0) =0$, $x_k(\cdot)$ is a solution of \eqref{main} for every $k\in\mathbb{N}$. In addition, for any $k\in\mathbb{N}$, we have 
	\begin{align*}
		\sup\limits_{t\in[0,T]}\vert x(t)-x_k(t)\vert & = \sup\limits_{0< t\leq\frac{1}{k}}\left\vert t^2\sin\left(\frac{1}{t^2}\right)-\frac{t}{k}\sin(k^2)\right\vert\\
		&\leq \sup\limits_{0 < t\leq\frac{1}{k}}\left\vert t^2\sin\left(\frac{1}{t^2}\right)\right\vert + \sup\limits_{0 < t\leq\frac{1}{k}}\left\vert \frac{t}{k}\sin(k^2)\right\vert\\
		&\leq \sup\limits_{0 < t\leq\frac{1}{k}}t^2+ \sup\limits_{0 < t\leq\frac{1}{k}}\frac{t}{k}\\
		&=\frac{2}{k^2}.
	\end{align*}
	Therefore, $x_k$ strongly converges to $x$ in $\mathcal{C}^0([0,T],\mathbb{R})$ as $k\to\infty$. However, since $x(\cdot)$ is not of bounded variation (see \cite[Problem~2, p.~331]{kolmogorov1975}), it is not absolutely continuous. Hence, $x$ is not a solution of \eqref{main}. We have thus shown that ${\rm Sol}(P,u_0)$ is non-closed in $\mathcal{C}^0([0,T],{\mathcal H})$.
\end{proof}

Next, we will prove that the solution set of~\eqref{main} is closed if it is regarded as a subset of an appropriate space. More precisely, the following theorem confirms that the Sobolev space $W^{1,1}([0,T],{\mathcal H})$ is such a space. (This result can be explained by the well known fact that the norm of  $W^{1,1}([0,T],{\mathcal H})$ is finer than the one of $\mathcal{C}^0([0,T],{\mathcal H})$.)

\begin{theorem}\label{closedness_thm}
	If the assumption~{\rm\ref{closedconvexassumption}} is satisfied then, for any $u_0\in C(0)$, the solution set ${\rm Sol}(P,u_0)$ is closed in $W^{1,1}([0,T],{\mathcal H})$.
\end{theorem}
\begin{proof}
	 Let $u_0\in C(0)$ be given. Suppose that  $\{u_k\}\subset {\rm Sol}(P,u_0)$ is a sequence converging strongly in $W^{1,1}([0,T],{\mathcal H})$ to $u$ as $k\to\infty$. Then, $u$ is an absolutely continuous function. To prove that $u$ satisfies the initial condition in~\eqref{main}, we can argue as follows. Since the norm in $W^{1,1}([0,T],{\mathcal H})$ is given by~\eqref{norm_in_W1,1}, we have 
	 \begin{equation}\label{L1_norm_u_k}
	 	\lim\limits_{k\to\infty}\displaystyle\int_0^T\Vert u_k(\tau)-u(\tau)\Vert d\tau =0
	 \end{equation} and 
	 \begin{equation}\label{L1_norm_dot_u_k}
	 	\lim\limits_{k\to\infty}\displaystyle\int_0^T\Vert \dot{u}_k(\tau)-\dot{u}(\tau)\Vert d\tau = 0.
	 \end{equation} 
	 Note that $u_k(t)=u_k(0)+\displaystyle\int_0^{t}\dot{u}_k(\tau)d\tau$ and $u(t)=u(0)+\displaystyle\int_0^{t}\dot{u}(\tau)d\tau$ for every $t\in [0,T]$ and for all $k\in\mathbb N$ (see~\cite[Remark~3.4(c)]{ATY_2021}). Hence, from~\eqref{L1_norm_u_k}, \eqref{L1_norm_dot_u_k}, and \cite[Theorem~4, p.~46]{diestel1977} it follows that 
	 \begin{align*}
	 	0&=\lim\limits_{k\to\infty}\displaystyle\int_0^T\Vert u_k(\tau)-u(\tau)\Vert d\tau\\
	 	&=\lim\limits_{k\to\infty}\left[\displaystyle\int_0^T\Big\Vert u_k(0)-u(0) + \displaystyle\int_0^\tau (\dot{u}_k(s)- \dot{u}(s))ds\Big\Vert d\tau\right]\\
	 	&\geq \liminf\limits_{k\to\infty}\left[\displaystyle\int_0^T\left(\Vert u_k(0)-u(0) \Vert -\left\|\displaystyle\int_0^\tau (\dot{u}_k(s)- \dot{u}(s))ds\right\|\right) d\tau\right]\\
	 	&\geq \liminf\limits_{k\to\infty}\left[\displaystyle\int_0^T\left(\Vert u_k(0)-u(0) \Vert -\displaystyle\int_0^T \|\dot{u}_k(s)- \dot{u}(s)\|ds\right) d\tau\right]\\
	 	& = \liminf\limits_{k\to\infty}\left[T\Vert u_0-u(0) \Vert-T\displaystyle\int_0^T \Vert \dot{u}_k(s)- \dot{u}(s)\Vert ds\right]\\
	 	&=T\Vert u_0-u(0) \Vert.
	 \end{align*}
	 So, $u(0)=u_0$.  
	 
	 It remains to prove that $u$ satisfies the differential inclusion in~\eqref{main}. 
	 
	 Setting ${\mathcal C}=\left\{\varphi\in L^1([0,T],{\mathcal H}) \mid \varphi(t)\in C(t) \text{ a.e.} \; t\in[0,T]\right\},$  we will prove that ${\mathcal C}$ is closed in $L^1([0,T],{\mathcal H})$. Let $\{\varphi_m\}\subset D$ be a sequence converging strongly in $L^1([0,T],{\mathcal H})$ to a function $\psi$. Thanks to Lemma~\ref{stronglp_pointwise}, we can find a subsequence $\{\varphi_{m_j}\}$ of $\{\varphi_m\}$ such that $\varphi_{m_j}(t)$ converges to $\psi(t)$ for almost every $t\in[0,T]$. Since $\varphi_{m_j}(t)\in C(t)$ a.e. $t\in[0,T]$ and $C(t)$ is closed, we have $\psi(t)\in C(t)$ a.e. $t\in[0,T]$. Hence, one has $\psi\in {\mathcal C}$. This shows that ${\mathcal C}$ is closed in $L^1([0,T],{\mathcal H})$. 
	 
	 Since $\{u_k\}\subset {\rm Sol}(P,u_0)$, we have $\dot{u}_k\in {\mathcal C}$ for all $k\in\mathbb N$. From~\eqref{L1_norm_dot_u_k} it follows that $\dot{u}\in {\mathcal C}$. So, $\dot{u}(t)\in C(t)$ for almost every $t\in[0,T]$. As $C(t)$ is convex for all $t\in[0,T]$, the inclusion  $A_1 \dot{u}_k(t)+A_0 u_k(t)-f(t) \in -{\mathcal N}_{C(t)}(\dot{u}_k(t))$ is equivalent to
	 \begin{equation}\label{VI_uk_simple}\langle A_1 \dot{u}_k(t)+A_0 u_k(t)-f(t), \dot{u}_k(t)-z\rangle\leq 0\quad\forall z\in C(t).
	 \end{equation} For each $k\in\mathbb N$, \eqref{VI_uk_simple} holds for almost every $t\in [0,T]$. Thus, there exists a subset $D_k\subset [0,T]$ having zero Lebesgue measure that \eqref{VI_uk_simple} holds for every $t$ in $[0,T]\setminus D_k$. Putting $D=\bigcup_{k\in\mathbb N}D_k$, we see that $D$ is a set of zero Lebesgue measure and \eqref{VI_uk_simple} holds for all $k\in\mathbb N$ and for every $t$ in  $[0,T]\setminus D$. For each $t$ from $[0,T]\setminus D$, passing the inequality in~\eqref{VI_uk_simple} to the limit yields
$$\langle A_1 \dot{u}(t)+A_0 u(t)-f(t), \dot{u}(t)-z\rangle\leq 0 \quad\forall z\in C(t).$$
Thus, for almost every $t\in[0,T]$, one has $A_1 \dot{u}(t)+A_0 u(t)-f(t) \in -{\mathcal N}_{C(t)}(\dot{u}(t))$. 

We have thus proved that $u\in{\rm Sol}(P,u_0)$ and, therefore, established the desired closedness of ${\rm Sol}(P,u_0)$ in $W^{1,1}([0,T],\mathcal{H})$.
\end{proof}

\section{Convexity of the Solution Set}\label{section_SP_convexity}
As the normal cone in the sense of convex analysis to a convex set can be presented in a variational way, sweeping processes and variational inequalities are closely related. So, the convexity of the solution set of a sweeping process may have some connections with that property of the solution set of a variational inequality. 

\begin{theorem}\label{convexity_thm_1}
	If the assumption~{\rm \ref{closedconvexassumption}} is fulfilled and $A_0=0$, then ${\rm Sol}(P,u_0)$ is convex for every $u_0\in C(0)$.
\end{theorem}
\begin{proof}
	Let $u_0\in C(0)$ be taken arbitrarily. It suffices to consider the case where ${\rm Sol}(P,u_0)$ is nonempty. Under the assumption~{\rm \ref{closedconvexassumption}} and the condition $A_0=0$, an absolutely continuous function $u$ belongs to ${\rm Sol}(P,u_0)$ if and only if $u(0)=u_0$ and
	$$\langle A_1 \dot{u}(t)-f(t), y-\dot{u}(t)\rangle\geq 0 \quad\forall y\in C(t)$$
	for a.e. $t\in[0,T]$. The latter means that $z(t):=\dot{u}(t)$ is a solution of the variational inequality
	\begin{equation}\label{VI_A1}
		\langle F(z,t),y-z\rangle \geq 0\quad\forall y\in C(t)
	\end{equation}
	for a.e. $t\in[0,T]$, where $F(z,t):=A_1z-f(t)$. By the assumed positive semidefiniteness of $A_1$, one has
	$$\langle F(z',t)-F(z,t),z'-z\rangle = \langle A_1(z'-z),z'-z\rangle \geq 0$$
	for every $z,z'\in\mathcal{H}$. Hence, $F(\cdot, t): \mathcal{H}\to\mathcal{H}$ is a monotone operator. Moreover, since the linear operator $A_1$ is bounded, $F(\cdot, t)$ is continuous. Therefore, applying Minty's lemma~\cite[Lemma~1.5]{KS_1980} for the monotone variational inequality~\eqref{VI_A1}, we can assert that the solution set of~\eqref{VI_A1} is closed an convex for every $t\in[0,T]$. Consequently, if $u$, $v$ are two elements of ${\rm Sol}(P,u_0)$ and $\lambda\in(0,1)$ is given arbitrarily, $(1-\lambda)\dot u(t)+\lambda \dot v(t)$ is a solution of~\eqref{VI_A1} for almost every $t\in[0,T]$.  Since $t\mapsto(1-\lambda)\dot{u}(t)+\lambda\dot{v}(t)$ is Bochner integrable (see~\cite[Proposition 1.4.17]{cazenave_haraux_1998}), the formula $w(t):=u_0+\displaystyle\int_0^t\left[(1-\lambda)\dot{u}(\tau)+\lambda\dot{v}(\tau)\right]d\tau$ defines an absolutely continuous function. Clearly, $w(0)=u_0$. In addition, we have $\dot{w}(t)=(1-\lambda)\dot{u}(t)+\lambda\dot{v}(t)$
	for a.e. $t\in[0,T]$ (see, e.g.,~\cite[Remark~3.4(d)]{ATY_2021}). So, $w(t)$ is a solution of~\eqref{VI_A1} for a.e. $t\in[0,T]$. This implies that 
	$$A_1 \dot{w}(t)+A_0 w(t)-f(t) \in -{\mathcal N}_{C(t)}(\dot{w}(t))\quad\text{a.e. }t\in[0,T].$$
	Hence, $w\in {\rm Sol}(P,u_0)$. The convexity of ${\rm Sol}(P,u_0)$ has been proved.
\end{proof}

The kernel of the operator $A_0:\mathcal{H}\to\mathcal{H}$ plays an important role in the forthcoming results. Recall that ${\rm ker \ } A_0:=\{x\in\mathcal{H}\mid A_0x=0\}$. Note that the quadratic form $\varphi(y):=\langle A_0y,y\rangle$ is Fr\'echet differentiable on $\mathcal{H}$ because $A_0$ is bounded (see, e.g.,~\cite[Proposition~2.1]{Yen_Yang_JOTA2018}). Since $\langle A_0y,y\rangle\geq 0$ for all $y\in\mathcal{H}$, a vector $x\in\mathcal{H}$ satisfies the equality $\langle A_0x,x\rangle=0$ if and only if $x$ solves the optimization problem $\min\{\varphi(y)\mid y\in\mathcal{H}\}$. If  $x$ is a solution of the latter, then by the Fermat rule one has $\nabla\varphi(x)=0$, i.e., $A_0x=0$. Conversely, if $A_0x=0$ then $\varphi(x)=0$. Therefore, we have 
\begin{equation}\label{ker_A0}
	\left\{x\in\mathcal{H}\mid \langle A_0x,x\rangle=0\right\} ={\rm ker \ }A_0.
\end{equation}

Under a mild assumption, using one solution $u$ of~\eqref{main}, we can construct a closed convex set $\mathcal{K}$ in $W^{1,1}([0,T],\mathcal{H})$, such that the solution set ${\rm Sol}(P,u_0)$ is contained in $u+\mathcal{K}$. Thus, the closed convex set $u+\mathcal{K}$ is an outer estimate for ${\rm Sol}(P,u_0)$. The estimate is sharp, because in some cases it holds as an equality (see Theorem~\ref{convexity_thm_2} below).    
 
\begin{theorem}\label{sol_character}
Suppose that~{\rm\ref{closedconvexassumption}} is satisfied. For any $u_0\in C(0)$, if ${\rm Sol}(P,u_0)$ is nonempty and $u$ is a selected solution of~\eqref{main}, then 
\begin{equation}\label{sol_subset_uK}
	{\rm Sol}(P,u_0) \subset u+\mathcal{K},
\end{equation}
where 
\begin{equation}\label{set_K}
	\mathcal{K}:= \left\{y\in W^{1,1}([0,T],\mathcal{H})\mid y(0)=0,\ \dot{y}(t)\in (C(t)-\dot{u}(t))\cap {\rm ker \ }A_0\; \;{\rm a.e. \ }t\in[0,T]\right\}
\end{equation} is a closed convex set.
\end{theorem}

\begin{proof} Select a solution $u$ of~\eqref{main}. Let $v\in{\rm Sol}(P,u_0)$ be chosen arbitrarily. Since~{\rm\ref{closedconvexassumption}} is fulfilled, we have
		\begin{equation*}
	\begin{cases}
	\langle A_1 \dot{u}(t)+A_0 u(t)-f(t), \dot{u}(t) -z\rangle \leq 0\quad \forall z\in C(t),\\
	\langle A_1 \dot{v}(t)+A_0 v(t)-f(t), \dot{v}(t) -z\rangle \leq 0\quad \forall z\in C(t)
	\end{cases}
	\end{equation*}
	for a.e. $t\in [0,T]$. As $\dot{u}(t)$ and $\dot{v}(t)$ belong to $C(t)$ for almost every $t\in [0,T]$, the latter implies that
	$$\langle A_1 \dot{u}(t)+A_0 u(t)-f(t), \dot{u}(t) - \dot{v}(t)\rangle \leq 0$$
	and
	$$\langle A_1 \dot{v}(t)+A_0 v(t)-f(t), \dot{v}(t) - \dot{u}(t) \rangle \leq 0$$ for a.e. $t\in[0,T]$. From the last inequalities one gets 
	$$\langle A_1 (\dot{u}(t)-\dot{v}(t))+A_0(u(t)- v(t)), \dot{u}(t) - \dot{v}(t)\rangle \leq 0$$
	for a.e. $t\in [0,T]$. As $A_1$ is positive semidefinite, it follows that
	$$\langle A_0(u(t)- v(t)), \dot{u}(t) - \dot{v}(t)\rangle \leq 0$$
	for a.e. $t\in [0,T]$. Integrating both sides of the last inequality and applying~\cite[Remarks~11.23(c)]{rudinprinciples} yield
	$$	\int_0^t\langle A_0 (u(\tau)- v(\tau)),  \dot{u}(\tau)- \dot{v}(\tau)\rangle d\tau\leq 0\quad \forall t\in [0,T].$$	
	As it has been noted in the proof of Theorem~\ref{continuity_dependence_1}, this implies
	$$\langle A_0 (u(t)-v(t)), u(t)- v(t)\rangle- \langle A_0 (u(0)-v(0)), u(0)- v(0)\rangle \leq 0\quad \forall t\in [0,T].$$
	Since $u(0)=v(0)$, the latter means that 
	$\langle A_0 (u(t)-v(t)), u(t)- v(t)\rangle \leq 0$ for all $t\in [0,T].$
	So, by the positive semidefiniteness of $A_0$, we obtain 
	$$\langle A_0 (u(t)-v(t)), u(t)- v(t)\rangle =0\quad \forall t\in [0,T].$$
	Therefore, setting $x(t):=v(t)-u(t)$, $t\in[0,T]$, by~\eqref{ker_A0} we have $x(t)\in{\rm ker \ }A_0$ for all $t\in[0,T]$. It is clear that $x(0)=v(0)-u(0)=0$ and $\dot{x}(t)=\dot{v}(t)-\dot{u}(t)\in C(t)-\dot{u}(t)$ for a.e. $t\in[0,T]$. Since $x(\cdot)$ is an absolutely continuous function, from the condition $A_0x(t)=0$ for all $t\in[0,T]$ we deduce that $A_0\dot{x}(t)=0$ for a.e. $t\in[0,T]$. Hence, $\dot{x}\in\mathcal{K}$. We have thus shown that~\eqref{sol_subset_uK} is valid.
	 The convexity and closedness of $\mathcal{K}$ can be easily verified by using the convexity and closedness of $C(t)$ for all $t\in [0,T]$.
\end{proof}

In the next theorem, we investigate the convexity of the solution set in the case where $A_0\neq 0$.
\begin{theorem}\label{convexity_thm_2}
Suppose that~{\rm\ref{closedconvexassumption}} is satisfied, $A_1=0$, and $f(t)\perp{\rm ker \ }A_0$ (i.e., $\langle f(t),x\rangle =0$ for every $x\in{\rm ker \ }A_0$) for all $t\in[0,T]$. Then, ${\rm Sol}(P,u_0)$ is convex for every $u_0\in C(0)$.
\end{theorem}

\begin{proof}
	Let $u_0\in C(0)$ be given arbitrarily and $u$ be a solution of~\eqref{main}. By Theorem~\ref{sol_subset_uK}, the inclusion~\eqref{sol_subset_uK}, where the set $\mathcal{K}$ is defined in~\eqref{set_K}, holds. Take any $x\in\mathcal{K}$. Then, the function $v$ defined by setting $v(t)=u(t)+x(t)$, $t\in [0,T]$, is a solution of~\eqref{main}. Indeed, for almost every $t\in[0,T]$, one has $$\dot{v}(t)=\dot{u}(t)+\dot x(t)\in \dot{u}(t) + (C(t)-\dot{u}(t))=C(t).$$ Note that $v(0)=u(0)+x(0)=u_0.$ Since $\dot{x}(t)\in{\rm ker \ }A_0$ for a.e. $t\in[0,T]$, $x(0)=0$, and the linear operator $A_0$ is bounded, by~\cite[Proposition~1.4.22]{cazenave_haraux_1998} we have
		\begin{equation}\label{A_0x(t)} A_0x(t)=A_0\left(x(0)+\displaystyle\int_0^t \dot{x}(\tau)d\tau\right)=A_0\displaystyle\int_0^t \dot{x}(\tau)d\tau=\displaystyle\int_0^t A_0\dot{x}(\tau)d\tau=0
		\end{equation}
	for all $t\in [0,T]$. By $\Omega$ we denote the set of all  $t\in[0,T]$ where the derivatives $\dot{u}(t)$, $\dot{x}(t)$ exist, $\dot{x}(t)\in(C(t)-\dot{u}(t))\cap {\rm ker\ }A_0$, and $A_0 u(t)-f(t)\in-\mathcal{N}_{C(t)}(\dot{u}(t))$. By our assumptions, $\Omega$ is a subset of full measure of $[0,T]$. For any $t\in\Omega$ and for any $z\in C(t)$, by~\eqref{A_0x(t)} we have
	\begin{equation*}
		\begin{array}{rl}
		\langle A_0 v(t)-f(t), z- \dot{v}(t)\rangle&=\langle A_0(u(t)+x(t))-f(t), z- (\dot{u}(t)+\dot{x}(t))\rangle\\
		&=\langle A_0u(t)-f(t), z- (\dot{u}(t)+\dot{x}(t))\\
		&=\langle A_0 u(t)-f(t), z- \dot{u}(t)\rangle -\langle A_0 u(t),\dot{x}(t)\rangle + \langle f(t),\dot{x}(t)\rangle\\
		&=\langle A_0 u(t)-f(t), z- \dot{u}(t)\rangle -\langle u(t),A_0\dot{x}(t)\rangle + \langle f(t),\dot{x}(t)\rangle.
		\end{array}
	\end{equation*} 
	Since $\dot{x}(t)\in{\rm ker\ }A_0$ and $f(t)\perp{\rm ker\ }A_0$, it follows that $\langle u(t),A_0\dot{x}(t)\rangle=0$ and $\langle f(t),\dot{x}(t)\rangle=0$. Therefore,
	\begin{equation}\label{v_estimate}
		\langle A_0 v(t)-f(t), z- \dot{v}(t)\rangle=\langle A_0 u(t)-f(t), z- \dot{u}(t)\rangle.
	\end{equation}
	As $u\in{\rm Sol}(P,u_0)$, the right hand side of~\eqref{v_estimate} is nonnegative. Hence, from~\eqref{v_estimate} we can deduce that $\langle A_0 v(t)-f(t), z- \dot{v}(t)\rangle\geq 0$. Since $z\in C(t)$ is can be chosen arbitrarily, we get
	$$\langle A_0 v(t)-f(t), z- \dot{v}(t)\rangle\geq 0\quad \forall z\in C(t)$$
	 for all $t\in\Omega$. Equivalently, $A_0 v(t)-f(t)\in-\mathcal{N}_{C(t)}(\dot{v}(t))$ for all $t\in\Omega$. It follows that $v$ is a solution of~\eqref{main}. So, we have proved that $u+\mathcal{K}\subset{\rm Sol}(P,u_0)$. Combining this with~\eqref{sol_subset_uK} yields ${\rm Sol}(P,u_0)= u+\mathcal{K}$. Hence, the desired convexity of ${\rm Sol}(P,u_0)$ follows from the convexity of the set $u+\mathcal{K}$.
\end{proof}

In connection with Theorems~\ref{convexity_thm_1}--\ref{convexity_thm_2}, we would like to raise the following open questions.

\textbf{Question 1.} \textit{We wonder if the assumptions $A_1=0$ and $f(t)\perp{\rm ker \ }A_0$ for all $t\in[0,T]$ could be dropped in the formulation of Theorem~\ref{convexity_thm_2}? In other words, does estimate~\eqref{sol_subset_uK} hold as an equality just under the assumption~{\rm\ref{closedconvexassumption}}?}

\textbf{Question 2.} \textit{Is there any example showing that, under the assumption~{\rm\ref{closedconvexassumption}}, the solution set of~\eqref{main} could be nonconvex?}

\section{Conclusions}

For sweeping processes with convex velocity constraints, we have obtained several new results on the solution sensitivity with respect to the initial value, as well as the closedness, the boundedness, and the convexity of the solution set. In addition, an outer estimate for the solution set is also given. Hoping for further in-depth studies on the solution set, we have proposed two open questions. 

\vskip 6mm
\noindent{\bf Acknowledgments}

\noindent  The author is grateful to Prof. Samir Adly, Prof. Nguyen Khoa Son, and Prof. Nguyen Dong Yen for encouragement and valuable discussions on the subject. This research was supported by the project \textit{"Sweeping processes with velocity constraints"} (Code: ICRTM03$\_$2021.05) of the International Center for Research and Postgraduate Training in  Mathematics (ICRTM) under the auspices of UNESCO of Institute of Mathematics, Vietnam Academy of Science and Technology.


\begin{thebibliography}{99}
	
	\bibitem{siddiqimanchanda2002} A.H. Siddiqi, P. Manchanda, Variants of Moreau's sweeping process, Adv. Nonlinear Var. Inequal. 5 (2002), no. 1, 1--16.
	
	\bibitem{aht} S. Adly, T. Haddad, L. Thibault, Convex sweeping process in the framework
	of measure differential inclusions and evolution variational inequalities, Math. Program., 148 (2014), Ser. B, 5--47.
	
	\bibitem{duvautlions1976} G. Duvaut, J.-L. Lions, Inequalities in Mechanics and Physics, Springer-Verlag, Berlin, 1976.
	
	\bibitem{ak18} S. Adly, B.K. Le, On semicoercive sweeping process with velocity
	constraint, Optim. Lett., 12 (2018), 831--843.
		
	\bibitem{adlyhaddad2020} S. Adly, T. Haddad, On evolution quasi-variational inequalities and implicit state-dependent sweeping processes, Discrete Contin. Dyn. Syst., 6 (2020), 1791–1801.
	
	\bibitem{ATY_2021} S. Adly, N.N. Thieu, N.D. Yen, Some classes of nonconvex sweeping processes with velocity constraints, preprint, 2021.
	
	\bibitem{bounkhel2007} M. Bounkhel, Existence and uniqueness of some variants of nonconvex sweeping processes, J. Nonlinear Convex Anal., 8 (2007), 311--323. 
	
	\bibitem{JouraniVilches2019} A. Jourani, E. Vilches, A differential equation approach to implicit sweeping processes, J. Differential Equations, 266 (2019), 5168--5184.
		
	\bibitem{vilchesnguyen2020} E. Vilches, B.T. Nguyen, Evolution inclusions governed by time-dependent maximal monotone operators with a full domain, Set-Valued Var. Anal., 28 (2020), 569--581.
	
	\bibitem{adlyhaddad2018} S. Adly, T. Haddad, An implicit sweeping process approach to quasistatic evolution variational inequalities,	SIAM J. Math. Anal., 50 (2018),  761--778.	
	
	\bibitem{Yen_AMO95} N.D. Yen, H\"older continuity of solutions to a parametric variational inequality,  Appl. Math. Optim., 31 (1995), 245--255.
	
	\bibitem{diestel1977} J. Diestel, J.J. Uhl, Jr., Vector Measures, American Mathematical Society, Providence, R.I., 1977. 
	
	\bibitem{benyamini1998} Y. Benyamini, J. Lindenstrauss, Geometric Nonlinear Functional Analysis, American Mathematical Society, 1998.
	
	\bibitem{Mordukhovich_2006} B.S. Mordukhovich, Variational Analysis and Generalized Differentiation. Vol. I: Basic Theory. Vol. II:
	Applications, Springer, Berlin, 2006.
	
	\bibitem{Mordukhovich_2018} B.S. Mordukhovich, Variational Analysis and Applications, Springer Monographs in Mathematics, Springer, Cham, 2018.
	
	\bibitem{brezis2011} H. Brezis, Functional Analysis, Sobolev Spaces and Partial Differential Equations, Springer, 2011.
	
	\bibitem{yosida1980} K. Yosida, Functional Analysis, Springer-Verlag, Berlin-New York, 1980.
	
	\bibitem{cazenave_haraux_1998} T. Cazenave, A. Haraux, An Introduction to Semilinear Evolution Equations, Oxford University Press, New York, 1998.

	\bibitem{zeidler_2A_1990}E. Zeidler. Nonlinear functional analysis and its applications. II/A. Linear monotone operators, Springer-Verlag, New York, 1990.
	
	\bibitem{rudinprinciples} W. Rudin, Principles of Mathematical Analysis, Third edition, McGraw-Hill Book Co., New York, 1976. 
	
	\bibitem{kolmogorov1975} A.N. Kolmogorov, S.V. Fomin, Introductory Real Analysis, Courier Corporation, 1975.
	
	\bibitem{Yen_AMV87} N.D. Yen,  Implicit function theorems for set-valued maps, Acta Math. Vietnam., 12 (1987), No.~2,  17--28.
	
	\bibitem{KS_1980} D. Kinderlehrer, G. Stampacchia, An Introduction to Variational Inequalities and Their Applications, Academic Press, Inc., New York-London, 1980.
	
	\bibitem{Yen_Yang_JOTA2018} N.D. Yen, X. Yang, Affine variational inequalities on normed spaces, J. Optim. Theory Appl., 178 (2018), 36--55.	 
	
\end{thebibliography}
\end{document}